\documentclass[psamsfonts]{amsart}
\usepackage{times,amsthm,amsfonts,amscd}

\theoremstyle{plain} 
\newtheorem{lemma}{Lemma}
\newtheorem{proposition}[lemma]{Proposition}

\newtheorem*{mainproposition*}{Main Proposition}
\theoremstyle{definition}
\newtheorem*{hypotheses*}{Hypotheses for the Main Proposition}
\newtheorem*{question*}{Question}
\newtheorem*{notation}{Notation}
\newtheorem{definition}{Definition}
\newtheorem*{definition*}{Definition}
\newtheorem*{construction}{Construction}

\newtheorem{example}{Example}
\theoremstyle{remark} 
\newtheorem{remark}{Remark} 
\newtheorem*{remark*}{Remark}

\newcommand{\id}{\operatorname{id}}

\newcommand{\coder}{\operatorname{Coder}}
\newcommand{\coalg}{\operatorname{Coalg}}

 \newcommand{\C}{\mathbb{C}}
 
 \newcommand{\E}{e}
\newcommand{\ba}{\mathbf{a}}

\title{Homotopy Probability Theory II}

\author{Gabriel C. Drummond-Cole}
\address{Center for Geometry and Physics, 
Institute for Basic Science (IBS), Pohang 790-784, Republic of Korea}

\thanks{This material is based in part upon work
     supported by the National Science Foundation under Award
     No. DMS-1004625.}
\email{gabriel@ibs.re.kr}

\author{Jae-Suk Park}
\address{ \begin{enumerate}
  \item[1)] Center for Geometry and Physics, 
Institute for Basic Science (IBS), Pohang 790-784, Republic of Korea
\item[2)] Department of Mathematics, Pohang University of Science and 
Technology (POSTECH), Pohang 790-784, Republic of Korea
  \end{enumerate}}

\thanks{This work was supported in part by the 
Mid-career Researcher Program through NRF funded by the MEST (no. 2010-0000497)}

\email{jaesuk@postech.ac.kr}

\author{John Terilla}
\address{Department of Mathematics, The Graduate Center and Queens
  College, The City University of New
  York, USA}
\thanks{Thanks to the Simons Center for
     Geometry and Physics for providing an excellent working environment.}

\email{jterilla@qc.cuny.edu}

 \keywords{probability, cumulants, homotopy}

 \subjclass[2000]{55U35, 46L53, 60Axx}

\begin{document}

 \maketitle

\begin{abstract}
  This is the second of two papers that introduce a deformation
  theoretic framework to explain and broaden a link between homotopy
  algebra and probability theory.  This paper outlines how the
  framework can assist in the development of homotopy probability
  theory, where a vector space of random variables is replaced by a
  chain complex of random variables.  This allows the principles of
  derived mathematics to participate in classical and noncommutative
  probability theory.  A simple example is presented.
\end{abstract}

\section{Introduction}
This paper is the second of two papers that interprets and utilizes a
link between homotopy algebra and probability theory \cite{Park2011}
found while studying certain algebraic aspects of quantum field theory
\cite{JS4,JS5}.  In the prequel
\cite{HPT1}, the authors present a deformation theoretic framework for
studying maps between algebras that do not respect structure.  
The
framework for studying maps that do not respect structure applies to
probability theory in the following way.  Expectation value is a
linear map between a vector space $V$ of random variables and the
complex numbers that does not respect the product structure on $V$.
The failure of the expectation value to respect the products in the
space of random variables and the complex numbers can be processed to
give an infinite sequence of operations $\kappa_n:V^{\otimes n}\to
\C$. This sequence of operations assembles into an $A_\infty$ morphism
between two trivial $A_\infty$ algebras. The main proposition in
\cite{HPT1} is that the $A_\infty$ morphism obtained via this process
coincides with the cumulants of the initial probability space.

In the case that the product of random variables is commutative, there
is a similar construction of an infinite sequence of symmetric
operations $k_n:S^nV\to \C$ which assemble into an $L_\infty$
morphism. In the commutative context, the result is that the
$L_\infty$ morphism coincides with the classical cumulants of the
initial probability space.

The framework for studying maps which fail to preserve structure
applies just as well when the space of random variables $V$ is
replaced by a chain complex $C=(V,d)$.  This generalization is pursued
in this paper.

The idea of replacing a space of random variables by a chain complex
of random variables is not unprecedented.  For example, the
Batalin-Vilkovisky formalism \cite{BV1} in quantum field theory
involves a differential, called the BV differential, on the space of
random variables \cite{KC1, JS4, JS5}.  In this formalism, expectation value is
defined via the Feynmann path integral which is a chain map from the
space of random variables to $\C[[\hbar]]$.  The BV differential
provides a way to capture some computational techniques of physicists
as homological algebra.  This work on homotopy probability theory is
about adapting physicists' computational techniques directly to
probability theory.  The presence of Planck's constant $\hbar$ in
quantum field theory, specifically the manner in which the products of
observables and the BV operator interact with $\hbar$, makes quantum
field theory significantly more complicated than probability theory.
One may think of a hierarchy with the proposed homotopy probability
theory in the middle.
\begin{itemize}
\item {\em Classical and noncommutative probability theory} are
  concerned with the expected value of a vector space of random
  variables equipped with an associative binary product.
\item {\em Homotopy probability theory} is concerned with the expected
  values of a chain complex of random variables equipped with an
  associative binary product.
\item {\em Quantum field theory} is concerned with the
  $\C[[\hbar]]$-valued expected values of a chain complex of random
  variables, called observables, equipped with an associative binary
  product satisfying certain subtle algebraic conditions related to
  $\hbar$.
\end{itemize}

In this paper, the idea of working with chain complexes rather than
vector spaces is isolated and studied---only the middle theory above
is considered.  Also, to keep the abstract ideas presented here
connected to simple examples, the product of random variables will be
assumed to be commutative throughout.  So, $L_\infty$ algebras, rather
than $A_\infty$ algebras, are used.  The requisite modifications and
results contained in \cite{HPT1} are given in Section
\ref{sec:commutative}.  

The reasons that quantum field theory frequently involves a
differential have little to do with quantum fields.  Rather, the
differential is a tool used to handle symmetries by algebraic methods,
and this tool can be applied to ordinary probability theory.  In much
the same way that one may resolve relations in homological algebra by
replacing a module with a free resolution, one might extend a space of
random variables with relations among expectation values to a larger
chain complex of random variables where the relations are encoded in
the differential.  Then homotopical computations provide new ways to
compute invariants of interest, such as joint moments.  This idea is
illustrated with a toy example in Section \ref{sec:example}.  Of
course, there may be several ways for a given space of random
variables to be extended to a chain complex of random variables, so
care should be taken to understand which quantities are invariant of
the original space of random variables and do not depend on choices
made in a particular extension.  This leads to the concept of
\emph{homotopy random variables} (Definition \ref{def:hrv}) and their
\emph{moments} and \emph{cumulants} (Definitions \ref{def:moment} and
\ref{def:cumulant}).

Independent of any utilitarian advantages to replacing spaces of
random variables by chain complexes of random variables, homotopy
theoretic ideas provide attractive structural aspects to probability
theory.  Cumulants, for example, are interpreted as a homotopy
morphism and then tools in homotopy theory, both computational and
theoretical, can be brought to bear upon them.  As another example,
this point of view indicates how to construct a category of
probability spaces that is different than previous
constructions---there are, for example, many more morphisms.  
Moreover, once the framework 
for studying maps that do not respect
structure is applied to probability theory, there become 
ways for the ideas, language, and tools of probability theory to
participate in other areas that feature a non-structure preserving
map.  Period integrals of smooth projective hypersurfaces are one
such area where the participation seems to generate some new concepts \cite{PP} .

The authors would like to thank Tyler Bryson, Joseph Hirsh, Tom
LeGatta, and Bruno Vallette for many helpful discussions.

\section{Commutative homotopy probability spaces and $L_\infty$
  algebras}\label{sec:commutative}
\begin{definition}
  A {\em commutative homotopy probability space} is a triple
  $(C,\E,a)$, where $C=(V,d)$ is a chain complex over $\C$, $\E$ is a
  chain map $C\to(\C,0)$ (where $\C$ is concentrated in degree $0$),
  and $a$ is an associative product $S^2 V\to V$.   $SV$ denotes the
  graded symmetric algebra so that $a$ is graded symmetric
  \[a(X_1,X_2) = (-1)^{|X_1||X_2|}a(X_2,X_1).\]
\end{definition}
\begin{remark}
  Note that $a$ is not required to possess any sort of compatibility
  with $d$ or $\E$ (which it will not in interesting examples).
\end{remark}
\begin{remark}
  There are multiple easy generalizations to this definition that this
  paper will not pursue.  It is natural to replace $\C$ with a
  different ground algebra.  It is also reasonable to allow $a$ to be
  any element of $\hom(SV,V)$. For example, $a$ might be a
  non-associative binary product, or a collection of non-binary products.
\end{remark}

\begin{definition}
  A {\em morphism} of commutative homotopy probability spaces between
  $(C,\E,a)$ and $(C',\E',a')$ is a chain map $f:C\to C'$ that
  commutes with expectation in the sense that $\E'f=\E$.
\end{definition}
\begin{remark}
  Note again that there is no compatibility assumed between $f$ and
  the products $a$ and $a'$.
\end{remark}

Basic facts and definitions about $L_\infty$ algebras are now
recalled.  For more details, see \cite{LV}.  Let $V$ be a graded
vector space.  Let $S^n V$ be the
$S_n$-invariant subspace of $V^{\otimes n}$ and let
$SV=\oplus_{n=1}^\infty S^nV$.  As a direct sum, linear maps from $SV$
to a vector space $W$ correspond to collections of linear maps $\{S^nV
\to W\}_{n=1}^\infty$.  Also, $SV$ is a coalgebra, free in a certain
sense, so that linear maps from a commutative coalgebra
$\mathcal{C}\to V$ are in bijection with the following two sets:
$\{\text{coalgebra maps from $\mathcal{C}$ to $SV$}\}$ and
$\{\text{coderivations from $\mathcal{C}$ to $SV$}\}$.
\begin{definition}
  \emph{An $L_\infty$ algebra} is a pair $(V,D)$ where $V$ is a graded
  vector space and $D:SV \to SV$ is a degree one coderivation
  satisfying $D^2=0$.  \emph{An $L_\infty$ morphism} between two
  $L_\infty$ algebras $(V,D)$ and $(V',D')$ is a differential
  coalgebra map $F:(SV,D)\to (SV',D')$.  In other words, an $L_\infty$
  map from $(V,D)$ to $(V',D')$ is a degree zero coalgebra map $F:SV
  \to SV'$ satisfying $FD=D'F$.
\end{definition}
\begin{remark}
Often, the definition of an $L_\infty$ algebra involves a degree shift 
(so according to that convention, a degree one coderivation $D:SV \to SV$ 
satisfying $D^2=0$ would define an
$L_\infty$ algebra on the underlying vector space $V[-1]$).  For the
applications to probability theory proposed in this paper, the
conventional degree shift makes the signs much more complicated---it's simpler to
eliminate the shift in the definition.
\end{remark}
The identification
\begin{equation}
  \label{eq:coder}
  \coder(SV,SV)\simeq
  \prod_{n=1}^\infty \hom(S^nV,V)
\end{equation}
can be used to give the data of an $L_\infty$ algebra.  That is, a
coderivation $D:SV \to SV$ can be given by a sequence $\{d_n\}$ of
degree one linear maps $d_n:S^nV \to V$ on a graded vector space $V$.
The condition that $D^2=0$ implies an infinite number of relations
satisfied by various compositions among the $\{d_n\}$.  Likewise, the
identification
\begin{equation}
  \label{eq:coalg}
  \coalg(SV,SV')\simeq \prod_{n=1}^\infty \hom(S^nV,V')
\end{equation}
can be used to give the data of an $L_\infty$ morphism between $(V,D)$
and $(V',D')$.  That is, a coalgebra map $F:SV \to SV'$ can be given
by a sequence $\{f_n\}$ of degree zero maps $f_n:S^nV \to V'$.  The
condition that $FD=D' F$ encodes an infinite number of relations among
the $\{f_n\}$, the $\{d_n\}$ and the $\{d'_n\}$.  So, the
identifications in Equations \eqref{eq:coder} and \eqref{eq:coalg}
provide two equivalent ways of describing $L_\infty$ algebras and
morphisms and it is convenient to move between the ways since certain
notions or computations are easier to describe in one or the other
descriptions of the equivalent data.  For example, it is
straightforward to compose two differential coalgebra maps
$F:(SV,D)\to (SV',D')$ and $G:(SV',D')\to (SV'',D'')$ as $GF:(SV,D)\to
(SV'',D'')$ and thus define the composition of $L_\infty$ morphisms,
but it is more involved to express the components $(gf)_n:S^nV \to
V''$ in terms of the $f_k:S^kV \to V'$ and $g_m:S^mV' \to V''$.
\begin{construction}
  Let $(V,D)$ be an $L_\infty$ algebra with components $d_n$. Let
  $A=(W,d)$ be a commutative differential graded algebra. Then
  $(V\otimes W)$ can be given a canonical $L_\infty$ algebra structure
  denoted $(V_A, D_A)$ whose first component is $d_1\otimes
  \id+\id\otimes d$ and whose $n$th component for $n>1$ are given by 
$$(v_1\otimes w_1, \ldots,
  v_n\otimes w_n)\mapsto d_n(v_1,\ldots, v_n)\otimes w_1\cdots w_n.$$ In particular, if $A=(\C,0)$ with the usual
  multiplication, then $(V_A,D_A)$ is canonically isomorphic to
  $(V,D)$.
\end{construction}

\begin{definition}
  Let $(V,D)$ and $(V',D')$ be $L_\infty$ algebras. Let $\Omega$
  denote the commutative differential graded algebra $\C[t,dt],$
  polynomials in a variable $t$ of degree $0$ and its differential
  $dt$ (so in particular $(dt)^2=0$). An {\em $L_\infty$ homotopy} $H$
  from $(V,D)$ to $(V',D')$ is an $L_\infty$ morphism from $(V,D)$ to
  $(V'_\Omega,D'_\Omega)$. Given an $L_\infty$ homotopy $H$, the
  evaluation maps $\Omega\to \C$ for $t=0$ and $t=1$ induce two
  $L_\infty$ morphisms from $(V,D)$ to $(V',D')$. Call these two
  $L_\infty$ morphisms $f$ and $f'$.  We say that $H$ is a homotopy
  {\em between} $f$ and $f'$.
\end{definition}
\begin{remark}
  Homotopy is an equivalence relation on $L_\infty$ morphisms from
  $(V,D)$ to $(V',D')$. Transitivity is not obvious but can be
  verified~\cite{Markl}.
\end{remark}
\begin{remark}
  Homotopy behaves well with respect to composition. That is, if $f$
  and $f'$ are homotopic $L_\infty$ maps $P\to Q$ and $g$ and $g'$ are
  homotopic $L_\infty$ maps $Q\to R$, then $g\circ f$ and $g'\circ f'$
  are homotopic $L_\infty$ maps $P\to R$. In particular, this is true
  in the special cases where $f=f'$ or when $g=g'$.
\end{remark}

Note that $L_\infty$ structures can be transported via isomorphisms.
In particular, if $(V,D)$ is an $L_\infty$ algebra and $G:SV \to SV$
is any degree zero coalgebra isomorphism, then for \[D^G:=G^{-1}DG\]
the pair $(V,D^G)$ is again an $L_\infty$ algebra and the map
$G:(SV,D^G)\to (SV,D)$ is an $L_\infty$ algebra isomorphism:
\[
 \begin{CD}
    SV @> D^G >> SV\\
    @V G VV @VV G V\\
    SV @> D >> SV
  \end{CD}
\]

Morphisms and homotopies can be transported as well.  If $F$ is an
$L_\infty$ morphism between $(V,D)$ and $(V', D')$ and $G:SV \to SV$
and $G':SV' \to SV'$ are coalgebra isomorphisms,
then \[F^{G,G'}:=(G')^{-1}FG\]  is an $L_\infty$ morphism between the $L_\infty$ algebras $(V,D^G)$
and $(V',(D')^{G'})$.  Here's the diagram:
\[
  \begin{CD}
    SV,D^G @> F^{G,G'} >> SV',D'^{G'}\\
    @V G VV @VV G' V\\
    SV,D @> F >> SV',D'
  \end{CD}
  \]

For a homotopy $H:SV \to SV'\otimes \Omega$, the map $G':SV'\to SV'$
must be modified to be a coalgebra isomorphism from $S(V'\otimes \Omega)$ to $S(V'\otimes
\Omega)$ but by abuse of notation the transfered homotopy $SV\to
S(V'\otimes \Omega)$ will still
be refered to as $H^{G,G'}$.


\begin{lemma}\label{lemma:nohomotopies}
  Let $(V,0)$ be a {\em trivial} $L_\infty$ algebra. If $f$ and $f'$
  are homotopic $L_\infty$ morphisms $(V,0)\to (\C,0)$, then $f=f'$.
\end{lemma}
\begin{proof}
  Let $H$ be a homotopy between $f$ and $f'$. Then $H$ is an
  $L_\infty$ morphism from $(V,0)$ to $\Omega$. That is, $H$ is a
  coalgebra map $SV\to S\Omega$ satisfying $dH=0$. In particular, each
  of the constituent maps $S^nV\to \Omega$ must have closed image in
  $\Omega$, so they must actually land in $\C\oplus dt\C[t,dt]$. This
  means that the evaluation of $H$ at $0$ and at $1$ yield the same
  map to $\C$, so $f$ and $f'$ coincide.
\end{proof}

\section{Expectation, moments, and cumulants}
\begin{notation}
  Let $(C,\E,a)$ be a commutative homotopy probability space, with
  $C=(V,d)$.  Then $(V,d)$ and $(\C,0)$ are $L_\infty$ algebras with
  vanishing higher maps and $\E$ is an $L_\infty$ morphism from
  $(V,d)$ to $(\C,0)$.   Let $\ba$ denote the isomorphism of coalgebras
  $SV\to SV$ whose $n$th component $S^nV\to V$ is repeated
  multiplication using $a$.
  Let $\mathbf{a}'$ denote the isomorphism of coalgebras
  $S\C\to S\C$ whose $n$th component $S^n\C\to \C$ is repeated complex
  multiplication.  The lowest component of $\ba$ and $\ba'$ are the identity.
\end{notation}
\begin{remark}
  In the case of not necessarily commutative homotopy probability
  theory, where $a$ is not necessarily commutative and $L_\infty$ is
  replaced throughout with $A_\infty$, the product $a$ can be extended
  as a coderivation on $TV=\oplus_{n=1}^\infty V^{\otimes n}$.  The
  coalgebra isomorphism $\ba$, whose components $V^{\otimes n}\to V$
  are repeated multiplication using $a$, is the exponential of the
  extended coderivation \cite{HPT1}.  In the commutative world the map
  $\ba:SV\to SV$ is defined to be the coalgebra isomorphism whose
  components $S^nV\to V$ are repeated multiplications using $a$, but
  the coalgebra isomorphism $\ba$ is not the exponential of the
  coderivation lift of of $a$.
\end{remark}

\begin{definition}
  Let $(C,\E,a)$ be a commutative homotopy probability space with
  $C=(V,d)$. The \emph{total moment} of $(C,\E,a)$ is the $L_\infty$
  morphism \[M:=\E^{\mathbf{a},\id}\] from $(V,D^\mathbf{a})$ to
  $(\C,0)$. The {\em $n$th moment map $m_n$} is the $n$th component of
  the total moment $\E^{\mathbf{a},\id}$ which is a map $m_n:S^nV\to
  \C$.
\end{definition}
\begin{definition}
Let $(C,\E,a)$ be a commutative homotopy probability
  space with $C=(V,d)$.  The {\em total cumulant} of $(C,\E,a)$ is the
  $L_\infty$ morphism \[K:=\E^{\mathbf{a},\mathbf{a'}}\] from
  $(V,D^\mathbf{a})$ to $(\C,0)$. The {\em $n$th cumulant $k_n$} is
  the $n$th component of the total cumulant
  $\E^{\mathbf{a},\mathbf{a'}}$ which is a map $k_n:S^nV\to
  \C$.  \end{definition} Moments and cumulants of a commutative
homotopy probability space are homotopy invariant.
\begin{proposition}
  Let $C$ be a chain complex equipped with a product $a$. Let $\E$ and
  $\E'$ be homotopic chain maps $C\to(\C,0)$. Then the total moment
  $M$ of the commutative homotopy probability space $(C,\E,a)$ is
  homotopic to the total moment $M'$ of the commutative homotopy
  probability space $(C,\E',a)$.  Also, the total cumulant $K$ of the
  commutative homotopy probability space $(C,\E,a)$ is homotopic to
  the total cumulant $K'$ of the commutative homotopy probability
  space $(C,\E',a)$.
\end{proposition}
\begin{proof}
  The transferred homotopy $H^{\ba, \id}$ is a homotopy between the
  total moments.  The transfered homotopy $H^{\mathbf{a},\mathbf{a}'}$
  is a homotopy between the total cumulants.
\end{proof}

In applications, one would like to evaluate moments and cumulants on
random variables and obtain numbers.  Direct application of moments
and cumulants yield numbers that are not homotopy invariant.  Homotopy
random variables are now introduced---they have joint moments and
cumulants that are homotopy invariant.

\begin{definition}\label{def:hrv}
  A {\em collection of homotopy random variables} $(X_1,\ldots, X_n)$
  in a commutative homotopy probability space $(C,\E,a)$ is the
  homotopy class of an $L_\infty$ map $(\C^n,0)\to
  (V,D^\mathbf{a})$. The various factors of $\C^n$ may be in different
  degrees.  When $n=1$, we call such a collection a {\em homotopy
    random variable}.
\end{definition}

\begin{example}
  The simplest kind of homotopy random variable is an $L_\infty$
  morphism $X:(\C,0)\to (V,D^\ba)$ whose only nonzero component is a
  map $\C\to V$.  Such a map is determined by $1\mapsto x$ for an
  element $x\in V$.  Not any choice of element $x\in V$ will define a
  homotopy random variable.  The condition that $X$ be an $L_\infty$
  morphism is that $D^\ba X=0$, which encodes an infinite collection
  of conditions on $x$.  Namely, $dx=0$, $d_2^\ba(x,x)=0$, etc...
  More generally, a single homotopy random variable $X$ will have
  components $S^k\C \to V$, each of which is given by a map
  $S^k1\mapsto x_k$ for some $x_k\in V$.  So, an arbitrary single
  random variable can be thought of as a sequence of elements
  $\{x_k\in V\}$.  Note that every closed element $x\in V$ gives rise
  to a homotopy random variable by transport.  That is, if $x\in V$
  satisfies $dx=0$, then $1\mapsto x$ defines a chain map $f:(\C,0)\to
  (V,d)$.  Such a chain map is an $L_\infty$ morphism $F:(\C,0)\to
  (V,D)$.  The transport of $F$ to an $L_\infty$ map
  $F^{\id,\ba}:(\C,0)\to (V,D^\ba)$ is a homotopy random variable.
  Note that not all homotopy random variables arise as the transport
  of chain maps; an $L_\infty$ morphism $X:(\C,0)\to (V,D^\ba)$ can be
  transported back to give $L_\infty$ morphisms
  $X^{\id,(\ba)^{-1}}:(\C,0)\to (V,D)$, but this map may not only be a
  chain map, there may be nonzero components $S^k\C \to V$ for $k>1$.
\end{example}

\begin{example}\label{ex:simplest}
  The simplest kind of collection of homotopy random variables is an
  $L_\infty$ morphism $(X_1, \ldots, X_n):(\C^n,0)\to (V,D^\ba)$ whose
  only nonzero component is a map $\C^n \to V$.  Such a morphism is
  determined by the images of the $i$th standard basis vectors of
  $\C^n$
  \[b_i\mapsto x_i\] for $i=1, \ldots, n$.  So, such a collection of
  homotopy random variables is determined by an ordered collection of
  elements $x_1, \ldots, x_n\in V$.  As in the case of a single
  homotopy random variable, not every collection of elements $\{x_1,
  \ldots, x_k\}$ will define a homotopy random variable and an
  arbitrary collection of homotopy random variables will generally
  have nonzero components $S^k\C^n\to V$ for $k>1$.  As in the case of
  a single homotopy random variable, finite collections of closed
  elements in $V$ define chain maps $f:(\C,0)\to (V,d)$ and can be
  transported to give collections of homotopy random variables.  Not
  all collections of homotopy random variables arise this way.
\end{example}

\begin{remark}
  Given a collection of homotopy random variables $(X_1,\ldots, X_n)$,
  one can obtain individual homotopy random variables $X_i$ by
  precomposing the inclusion of $\C$ into $\C^n$ in the $i$th
  factor. It is not true in general that given a set
  $\{X_i\}_{i=1,\ldots, n}$ of random variables one can meaningfully
  generate a collection of homotopy random variables $(X_1,\ldots,
  X_n)$.
\end{remark}

\begin{definition}\label{def:moment}
  Let $(C,\E,a)$ be a commutative homotopy probability space.
  The \emph{joint moment} $M(X_1, \ldots, X_n)$ of the collection of
  homotopy random variables $(X_1, \ldots, X_n)$ is defined to be the
  $L_\infty$ morphsism
  \[M(X_1, \ldots, X_n)=M \circ f \] where $f:(\C^n,0)\to (V,D^\ba)$
  is any representative of the collection of homotopy random variables
  $(X_1, \ldots, X_n)$ and $M$ is the total moment of $(C,\E,a)$.  The
  $r$th joint moment of a collection of random variables is the $r$th
  component of the joint moment $M(X_1, \ldots, X_n)$ and so is a map
  $m_r(X_1, \ldots, X_n):S^r\C^n\to \C$.
\end{definition}

\begin{definition}\label{def:cumulant}
  Let $(C,\E,a)$ be a commutative homotopy probability space.  The
  \emph{joint cumulant} $C(X_1,\ldots, X_n)$ of the collection of
  homotopy random variables $(X_1, \ldots, X_n)$ is defined to be the
  $L_\infty$ morphsism
  \[ K(X_1, \ldots, X_n)=K\circ f\] where $f:(\C^n,0)\to (V,D^\ba)$ is
  any representative of the collection of homotopy random variables
  $(X_1, \ldots, X_n)$ and $K$ is the total cumulant of $(C,\E,a)$.
  The $r$th joint cumulant of a collection of random variables is the
  $r$th component of the joint cumulant $C(X_1, \ldots, X_n)$ and so
  is a map $k_r(X_1, \ldots, X_n):S^r\C^n\to \C$.
\end{definition}

\begin{lemma}
  Joint moments and joint cumulants are well-defined.
\end{lemma}

\begin{proof}
  Let $f$ and $f'$ be two representatives of the homotopy class
  $(X_1,\ldots, X_n)$.  Then $M \circ f$ and $M\circ f'$ are homotopic
  $L_\infty$ maps from $(\C^n,0)$ to $(\C,0)$. By
  Lemma~\ref{lemma:nohomotopies}, these two maps coincide.  The
  compositions $K \circ f$ and $K\circ f'$ coincide for the same
  reason.
\end{proof}

\begin{remark}The $L_\infty$ structure $D^\ba$ does not depend on
  $\E$.  A collection of homotopy random variables could be defined
  for a pair $(C,a)$ without reference to an expectation.  In particular, a
  collection of homotopy random variables $(X_1, \ldots, X_n)$ for a
  commutative homotopy probability space $(C,\E,a)$ is also a
  collection of homotopy random variables for a commutative homotopy
  probability space $(C,\E',a)$.
\end{remark}

\begin{proposition}Let $C$ be a chain complex equipped with a product
  $a$. Let $\E$ and $\E'$ be homotopic chain maps $C\to(\C,0)$.  Let
  $(X_1, \ldots, X_n)$ be a collection of homotopy random variables of
  the commutative homotopy probability spaces $(C,\E, a)$ and $(C,\E',
  a)$.  Then the joint moments and joint cumulants coincide.  That is,
  \begin{gather*}
    m_r(X_1, \ldots, X_n)=m'_r(X_1, \ldots, X_n)\intertext{and}
    k_r(S_1, \ldots, X_n)=k_r'(X_1, \ldots, X_n).
  \end{gather*}
\end{proposition}
\begin{proof}
  Cumulants evaluated on a collection of homotopy random variables
  constitute $L_\infty$ morphisms from $(\C^n,0)\to (\C,0)$. The
  proposition shows that in the case in question, two such $L_\infty$
  maps are homotopic. By Lemma~\ref{lemma:nohomotopies}, they in fact
  coincide.
\end{proof}

To summarize, a collection of random variables $(X_1, \ldots, X_n)$ is
a homotopy class of $L_\infty$ maps $(\C^n,0)\to (V,D^\ba)$.  The joint
moments and joint cumulants of a collection of homotopy random
variables $\E(X_1, \ldots, X_n)$ are defined to be the compositions of
$L_\infty$ morphisms
\begin{gather}\label{eq:expect}
  (\C^n,0)\overset{(X_1, \ldots,
    X_n)}{\longrightarrow} (V,D^\ba)\overset{M}{\to}(\C,0)\\
  (\C^n,0)\overset{(X_1, \ldots, X_n)}{\longrightarrow}
  (V,D^\ba)\overset{K}{\to}(\C,0)
\end{gather}
These compositions are well defined.  They are independent of the
representative of $(X_1, \ldots, X_n)$ and only depend on the homotopy
class of the expectation value chain map $e:(V,d)\to (\C,0)$. As an
$L_\infty$ map $(\C^n,0)\to (\C,0)$, the joint moments and joint
cumulants of a collection of homotopy random variables $\E(X_1,
\ldots, X_n)$ have component maps
\[m_r:S^r(\C^n) \to \C \text{ and }k_r:S^r(\C^n) \to \C \] which are
comparable to the joint moments and cumulants of a collection of
random variables in classical probability theory.  In the special case
that a linear map $\C^n\to V$ defined by $b_i\mapsto x_i$ does in fact
define an $L_\infty$ morphism, as in Example \ref{ex:simplest}, then
the joint moment map is given by
\begin{align*}
  S^r\C^n &\to \C \\
  (b_{i_1}, \ldots, b_{i_r})&\mapsto \E(x_{i_1} \cdots x_{i_r})
\end{align*}
and the joint cumulant map is given by
\begin{align*}
  S^r\C^n &\to \C \\
  (b_{i_1}, \ldots, b_{i_r})&\mapsto k_r(x_{i_1}, \ldots, x_{i_r})
\end{align*}
where $k_r:S^kV\to \C$ is the $r$th cumulant of $(C,\E,a)$.

\section{The Gaussian}\label{sec:example}
The Gaussian in one variable provides a toy example of how one might
extend an algebra of random variables to a chain complex of random
variables.  The classical Gaussian in one variable is a classical
probability space $(V,\E,a)$ defined as follows: The space of random
variables $V$ is the space of polynomials in one variable,
concentrated in degree zero, the
expectation $\E(p(x))$ of a random variable is given by
\[\E(p(x))=\frac{\int_\mathbb{R} p(x)e^{-\frac{x^2}{2}}
  dx}{\int_\mathbb{R} e^{-\frac{x^2}{2}} dx},
\]
and the product $a:S^2V\to V$ is the ordinary product of polynomials.

Now, we construct a related homotopy probability space $(C,\E,a)$.
The idea behind the chain complex $(V,d)$ is to add new odd,
nilpotent, random
variables whose image under the differential comprises the kernel of
the expectation value $\E$.  Let $V=V^0\oplus V^1$ be the free
commutative algebra on two generators $x$ of degree $0$ and $\eta$ of
degree $-1$.  The product $a:V\otimes V\to V$ is the free graded commutative
product.  A typical element of $V$ has the form $p+q\eta$ for two
polynomials $p$ and $q$ in the variable $x$.  The $p$ term has degree
zero and the $q\eta$ term has degree $-1$.  The product is determined
by the fact that $\eta$ as an element of degree $-1$ squares to zero.
So \[(p+q\eta)(r+s\eta)=pr+(ps+qr)\eta.\] Define the differential
$d:V\to V$ by
\[d(p+q\eta)=\frac{dq}{dx}-xq.\] Define the expectation $\E:V\to \C$ by
\[\E(p+q\eta)=\frac{\int_\mathbb{R} pe^{-\frac{x^2}{2}}
  dx}{\int_\mathbb{R} e^{-\frac{x^2}{2}} dx}.
\]
Since \begin{align*}
  \E(d(p+q\eta))&=\E\left(\frac{dq}{dx}-xq\right)\\
  &=\frac{1}{\int_\mathbb{R} e^{-\frac{x^2}{2}} dx}\int_{\mathbb R} \left(\frac{dq}{dx}-xq \right)e^{-\frac{x^2}{2}} dx \\
  &=\frac{1}{\int_\mathbb{R} e^{-\frac{x^2}{2}} dx}\int_{\mathbb R}
  \frac{d}{dx} \left(qe^{-\frac{x^2}{2}}\right) dx\\
  &=0
\end{align*}
the expectation defines a chain map $\E:(V,d)\to (\C,0)$.  So if two
polynomials represent the same homology class in $C$, they have
the same expectation.  In fact, the converse is true: 
two polynomials have the
same expectation if and only if they represent the same homology class
in $C$.   So, $d$ encodes all the relations among
expectations and to compute the expectation of any one of the original
random variables, it suffices to know the homology class of the random
variable and the map $\E$ on homology.  Despite this fact, 
the map $\E$ on homology retains
almost no information of the original one variable Gaussian because
$\E$ on homology has no information about the product structure on the
space of random variables.   To
illustrate this point and the thinking behind homotopy random
variables, rephrase things in terms of
homotopies of maps.

If $f=f(x)$ and $g=g(x)$ are polynomials that represent the same homology
class in $(V,d)$, then the maps $1\mapsto f$ and $1\mapsto g$ define
homotopic chain maps $(\C,0)\to (V,d)$.  The corresponding probability
theory statement
is that their expectations 
agree $\E(f)=\E(g)$.   However, the moments of 
$f$ and $g$ will not agree in general, for example $\E(f^2)$ and
$\E(g^2)$ need not agree.
If, however, the maps $1\mapsto f$ and $1\mapsto g$ define homotopic
$L_\infty$ maps $(\C,0)\to (V,D^\ba)$, then not only will their
expectations agree, but all of their moments will agree as well.  
This can be verified explicitly.

First, the $L_\infty$ structure $D^\ba$ on $(V,d)$ can be computed.
Recall $D^\ba$ is the $L_\infty$ structure
on $(V,d)$ transported by the coalgebra automorphism $\ba:SV \to SV$
determined by the product $a:V\otimes V \to V$ of polynomials:
\[
  \begin{CD}
    SV @> D^\ba  >> S\C\\
    @V \ba VV @VV \ba' V\\
    SV @> D >> S\C
  \end{CD} 
\] 
The map $D:SV\to S\C$ along the bottom is the map $d:V\to \C$ lifted
as a coderivation.
Let $d_n^\ba:S^nV \to \C$ denote the components of $D^\ba$.   One way to
compute $d_n$ is to trace an element of $S^nV$ from the top left to
the bottom right along the two paths in the diagram above and compare
the part that lands in $\C \subset S\C$.
For $n=1$, this comparison gives $d_1^\ba =d$.  
For $n=2$, 
following the diagram along the top and right yields \begin{multline*}
(v_1, v_2)\mapsto d_2^\ba
(v_1, v_2)+(d(v_1) ,v_2)+(-1)^{|v_1|}(v_1, d(v_2))\\
\mapsto d_2^\ba(v_1,v_2)+d(v_1)v_2+(-1)^{|v_1|}v_1d(v_2)
\end{multline*}
and
along the left and bottom yields
\[
(v_1, v_2)\mapsto v_1v_2
\mapsto d(v_1v_2)
\]
Equating these and solving for $d_2^\ba$ gives 
\[d_2^\ba(v_1,v_2)=d(v_1v_2)-d(v_1)v_2-(-1)^{|v_1|}v_1d(v_2).\]
Therefore
\begin{align*}
  d^\ba_2(p+q\eta,r+s\eta)&=
d(pr+ps\eta+qr\eta)-d(p+q\eta)(r+s\eta)-pd(r+s\eta)+q\eta
d(r+s\eta)\\
&=
  \frac{dp}{dx}s+q\frac{dr}{dx}
  -\frac{dq}{dx}s\eta +q\frac{ds}{dx}\eta.
\end{align*}
The higher components $d^\ba_k=0$ for $k\geq 3.$   To see this, 
it is helpful to notice that \[d=\frac{\partial^2}{\partial x \partial \eta}-x
\frac{\partial}{\partial \eta}\] is a degree $\leq 2$ differential
operator.   Use the diagram as in the description of $d^\ba_2$ to express
$d_3^\ba(v_1,v_2,v_3)$ as the difference of $d(v_1v_2v_3)$ and a sum
of three terms of the form $d_2^\ba(v_i,v_j)v_k$ and three terms of
the form $d(v_i)v_jv_k$ to 
check that $d_3^\ba(v_1,v_2,v_3)=0$.   Then, assuming that
$d^\ba_n = 0$ for $3\le n < k$, 
we get that \begin{multline*}
d^\ba_k(v_1,\ldots, v_k)= 
d(v_1\cdots v_k) \\- \sum d_2^\ba(v_i v_j) v_1\cdots \widehat{v_i}\cdots
\widehat{v_j}\cdots v_k 
-\sum v_1\cdots d(v_i)\cdots v_k
\end{multline*}
which vanishes.

The $L_\infty$ algebra $(V,D^\ba)$ in this case is a differential
graded Lie algebra, the bracket defined by $d_2^\ba$ may be familiar
as the Poisson bracket defined on the polynomial functions defined on
an odd symplectic vector space.  Since both $d(f)=0$ and
$d_2^\ba(f,f)=0$ for a polynomial $f=f(x)$, the
map $1\mapsto f$ is an $L_\infty$ map $(\C,0)\to (V,D^\ba)$.   Because
joint moments of homotopy random variables are well defined, 
if two maps $1\mapsto f$ and $1\mapsto g$ are homotopic
$L_\infty$ maps then $\E(f^n)=\E(g^n)$ for all $n$.  

As a final computation,  let $X$ be the homotopy random variable defined by
$1\mapsto x$.  The joint moment of $X$
 is an $L_\infty$ morphism $MX:(\C,0)\to (V,D^\ba)\to
(\C,0)$ where the first map is $X$ and the second is
$M:(V,D^\ba)\overset{M}{\to}(\C,0)$.  
The components of $MX$ are maps
$m_k:S^k\C \to \C$ given by $S^k1 \mapsto \E(x^k)$.   As mentioned
earlier, the expectation of any one random variable, such as
$\E(x^k)$, can be computed by understanding the homology class of the
random variable and the map $\E$ in homology.  
Here, the homology of $(V,d)$ is one dimensional spanned by the class
of $1$ and 
\[\E(1)=\frac{\int_\mathbb{R} e^{-\frac{x^2}{2}} dx}{\int_\mathbb{R}
  e^{-\frac{x^2}{2}} dx}=1.\]  That determines $\E$ on homology.  To
determine the homology class of $x^k$, note, $d(-\eta)=x\Rightarrow [x]=[0]$ and for $n>
1$, $d(x^{n-1}\eta)=(n-1) x^{n-2}- x^{n}\Rightarrow [x^n]=(n-1)[x^{n-2}].$
So $[x^k]=[0]$ if $k$ is  odd and $[x^{2j}]=[(2j-1)!!]$ if $k=2j$
  is even.
Thus, the components of $MX$, where $X$ is the homotopy random variable defined by
$1\mapsto x$ and $M$ is the total joint moment of the Gaussian, 
are maps
$m_k:S^k\C \to \C$ given by $$m_k(S^k1)= \begin{cases}
0 & \text{ if $k$ is odd} \\
(2j-1)!! & \text{ for $k=2j$ even.}
\end{cases}
$$


\end{document}